\documentclass[a4paper]{article}

\usepackage{bold-extra}
\usepackage{arxiv}
\usepackage[english]{babel}
\usepackage{caption}
\usepackage{lscape} 
\usepackage{amsmath}
\usepackage[margin=2cm]{geometry}
\usepackage{amsthm}
\usepackage{adjustbox}
\usepackage{amsfonts}
\usepackage{mathtools}
\usepackage{verbatim}
\usepackage{algpseudocode,algorithm,algorithmicx}
\usepackage{mathtools}
\usepackage{comment}
\usepackage{verbatim, booktabs, dcolumn, setspace}

\newtheorem{theorem}{Theorem}

\newtheorem{lemma}{Lemma}

\newtheorem{remark}{Remark}
\usepackage{graphicx}
\usepackage{wrapfig}
\usepackage{caption}
\usepackage{mathtools}
\usepackage{float}

\allowdisplaybreaks

\usepackage{tikz}
\usetikzlibrary{matrix, positioning, arrows.meta, decorations.pathreplacing}

\usepackage[english,noautotitles-r]{SASnRdisplay} 
\usepackage[round]{natbib}
\usepackage{url}
\usepackage{subfigure}
\lstdefinestyle{r-output}{
style = r-style,
style = r-output-user,
}

\usepackage{url}            
\usepackage{booktabs}       
\usepackage{amsfonts}       
\usepackage{nicefrac}       
\usepackage{microtype}      
\usepackage{graphicx}
\usepackage{natbib}
\usepackage{doi}

\newcommand{\stirlingii}{\genfrac{\{}{\}}{0pt}{}}

\title{Extending an Order-Statistic Characterization of the Exponential Distribution}

\author{ \href{https://orcid.org/0000-0002-0460-400X}{\includegraphics[scale=0.5]{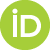}\hspace{1mm} Danijel G.    Aleksi\' c} \\
	University of Belgrade\\
	\scalebox{0.9}[1.0]{Faculty of Organizational Sciences}\\
	Belgrade, 11000, Serbia \\
	\texttt{\scalebox{0.9}[1.0]{danijel.aleksic@fon.bg.ac.rs}}\\
    \And 
    \href{https://orcid.org/0000-0002-9088-879X}{\includegraphics[scale=0.5]{OrcidLogo.eps}\hspace{1mm}Dušan Džamić} \\
	University of Belgrade\\
        \scalebox{0.9}[1.0]{Faculty of Organizational Sciences}\\
	Belgrade, 11000, Serbia \\
	\texttt{\scalebox{0.85}[1.0]{dusan.dzamic@fon.bg.ac.rs}} \\
    \And
    \href{https://orcid.org/0000-0001-8243-9794}{\includegraphics[scale=0.5]{OrcidLogo.eps}\hspace{1mm}Bojana Milo\v sevi\' c} \\
	University of Belgrade\\
        Faculty of Mathematics\\
	Belgrade, 11000, Serbia \\
	\texttt{bojana@matf.bg.ac.rs} \\
	\And \And
	\href{https://orcid.org/0000-0002-6826-3232}{\includegraphics[scale=0.5]{OrcidLogo.eps}\hspace{1mm}Marko Obradović} \\
	University of Belgrade\\
        Faculty of Mathematics\\
	Belgrade, 11000, Serbia \\
	\texttt{marcone@matf.bg.ac.rs} \\
}

\date{}

\begin{document}	
\maketitle

\begin{abstract}
We prove a novel characterization of the exponential distribution via a linear combination of order statistics. The theorem established herein generalizes the previous characterization theorem given in \cite{milovsevic2016some}. The proof relies on properties of Stirling numbers of the second kind, which are derived for this purpose and may be of independent interest.
    
    \vspace{10pt} \small

\end{abstract}

\section{Introduction}\label{sec:introduction}


The exponential distribution is the distribution for which the largest number of characterization theorems exist. Among the plethora of results proposed and proved by researchers, the memoryless property and the constancy of the hazard rate are undoubtedly the most famous.

A prominent class of characterizations relies on linear combinations of order statistics. These characterizations stem from the well-known identity
\begin{align}\label{identitet}
X_{(k;n)}\overset{d}{=}\sum_{j=n-k+1}^n \frac{X_j}{j}, \qquad 1\leq k\leq n,
\end{align}
which is known to hold for the exponential distribution \cite[see e.g.][p.~73]{arnold2008first}. \cite{ahsanullah1972characterization} proved that if \eqref{identitet} holds for all $k$, then it characterizes the exponential distribution. Later, \cite{huang1974characterizations} weakened this condition by showing that it suffices for the identity to hold for two consecutive values of $k$, and that a single value of $k$ is not sufficient in general.

Nevertheless, by imposing additional conditions on the distribution of the underlying i.i.d.\ sample, it becomes possible to obtain a characterization based on \eqref{identitet} holding for a single $k$.

\cite{arnold2013exponential} proposed a proof method that assumes the density is analytic for all $x>0$. Under this assumption, they established that the identity for the particular case $k=n=2$ characterizes the exponential distribution. They further conjectured that, under the same analyticity condition, the identity for any $k=n$ would also yield a characterization.

This work inspired several subsequent characterizations that employ the same proof technique. The identity \eqref{identitet} is rather flexible: various distributional equalities can be obtained when a part of the sum on the right hand side is replaced by the corresponding order statistic. For example,
for the  exponential distribution the following identity also holds  
\begin{align*}
    X_{(k;n)}\overset{d}{=}X_{(k-1;n-1)}+\frac{X_n}{n}.
\end{align*}
\cite{yanev2013characterizations} and \cite{obradovic2015three} studied some such particular cases for small $n$. The case $k=n$ for arbitrary $n$ was investigated by \cite{chakraborty2013characterization} and \cite{yanev2016characterization}, while \cite{milovsevic2016some} proved various characterizations for fixed but arbitrary $k$ and $n$, including one that directly matches the form of \eqref{identitet}.

A notable methodological innovation introduced in these later works is the use of Stirling numbers of the second kind in the proofs.

It is also worth mentioning that the conjecture stated by \cite{arnold2013exponential} was resolved as a corollary of the results obtained independently by both \cite{milovsevic2016some} and \cite{yanev2016characterization}.
Related characterizations were also considered in \cite{chakraborty2018characterizations}.

Here  we  prove   the  following theorem.


\begin{theorem}\label{main_theorem}
Let $\mathcal{F}$ denote the family of continuous cumulative distribution functions $F$ satisfying $F(0) = 0$, where the corresponding probability density function $f$, {such that $f(0)>0$,} admits a Maclaurin series expansion for every $x > 0$. Let $X_1, X_2, \dots, X_n$ be a random sample from the distribution $F$ that belongs to $\mathcal{F}$. Let $j$ and $k$ be fixed numbers such that {$1 \leq j < k < (n+j)/2 < n$} and let $X_{01}, X_{02}, ..., X_{0j}$ be independent random variables, each of them independent from sample $X_1, \dots, X_n$ and each from the distribution $F$. If
\[ X_{(k-j;n)} + \sum_{i=1}^j \frac{1}{n-k+i}X_{0i} \overset{d}{=} X_{(k;n)}, \]
then $F \sim \mathcal{E}(\lambda)$, for some $\lambda > 0$.
\end{theorem}

Its proof is given in the next section. In addition, we prove a combinatorial identity  featuring Stirling numbers of the second kind  that  might be useful in its own right.  The last section  contains an overview  of potential applications.

\section{Proof of Theorem \ref{main_theorem}}

Before establishing our main result, we present a series of auxiliary results which are known in the literature and are restated here for readers' convenience. The subsection ~\ref{subsec:auxiliary_lemmas}  contains  a series of new technical algebraic and combinatorial identities. Finally, the  proof is finalized  in subsection \ref{main_statement}.

\subsection{Preliminaries}\label{sec:preliminaries}

 The first two lemmas review standard properties and explicit representations of the Stirling numbers of the second kind. The remaining lemmas provide specific combinatorial identities and analytical characterizations of distribution functions required for the subsequent proofs.

\begin{lemma}[e.g. \cite{gross2007combinatorial}, Ch. 5]\label{stirling}
For Stirling number of the second kind it holds

\begin{equation}\label{stirling_1}
\stirlingii{a}{b} = \stirlingii{a-1}{b-1} + b \stirlingii{a-1}{b},
\end{equation}

\begin{equation}\label{stirling_2}
\stirlingii{a+1}{b+1} = \sum_{l=0}^{a}\binom{a}{l} \stirlingii{l}{b}, 
\end{equation}



\end{lemma}

\begin{lemma}[e.g. \cite{Charalambides2002}, Th. 8.6]
For Stirling numbers of the second kind it holds that
    \begin{align}\label{stirling_multinom}
\stirlingii{n}{k} = \frac{n!}{k!} \sum_{\substack{{r_1, \dots, r_k \geq 1} \\r_1 + r_2 + \cdots + r_k = n}} \frac{1}{r_1! r_2! \cdots r_k!}.
\end{align}
\end{lemma}

\begin{lemma}[\cite{milovsevic2016some}, Lemma 4
]\label{lema_induktivna_suma}
For integers $k, n, r$ such that $1\leq k \leq n$ and $r \geq 0$ it holds

\[ \sum_{\substack{j_1, \dots , j_k \geq 0 \\ j_1 + \cdots j_k = r+1}} n^{j_1}(n-1)^{j_2} \cdots (n-k+1)^{j_k} = \sum_{i=0}^{r+1} \binom{n-k}{i} \frac{(i+k-1)!}{(k-1)!} \stirlingii{k+r+1}{i+k}. \]
\end{lemma}

\begin{lemma}[\cite{milovsevic2016some}, Lemma 5]\label{lema_q}
Let $F$ be a distribution function that belongs to $ \mathcal{F}$. If for all natural $q$ holds
\begin{equation}\label{uslov_q}
f^{(q)}(0) = (-1)^qf^{q+1}(0),
\end{equation}
then $f(x) = \lambda e^{-\lambda x}$ for some $ \lambda > 0$.
\end{lemma}

\begin{lemma}[\cite{milovsevic2016some}, Lemma 6]\label{A_m_izraz}
Let $F$ be a distribution function that belong to the class $\mathcal{F}$. Denote $A_m(x) = F^{m}(x)f(x)$. If the condition \eqref{uslov_q} is satisfied for all $0 \leq q \leq r-m$, $r > m$, then
\begin{equation}\label{Am_izraz}
A_m^{(r)}(0) = (-1)^{r-m}f^{r+1}(0)\stirlingii{r+1}{m+1}m!,
\end{equation}
where $\stirlingii{a}{b}$ denotes the Stirling number of the second kind, i.e. number of ways to partition a set of $a$ elements into $b$ non-empty subsets.
\end{lemma}

\subsection{Auxiliary results}\label{subsec:auxiliary_lemmas}

Before proving the main theorem, we establish a few intermediate lemmas to handle the underlying algebraic and combinatorial complexities.

\begin{lemma}\label{lema_pomocna_za_pomocnu_dzamic}
For $1\leq m\leq N$ and all feasible real $x$, it holds that
    \[
\sum_{i=0}^{N-m}
\binom{N-m}{i}m^{\overline{i}}
\frac{x^{i}}{\prod_{q=1}^{m+i}(1-qx)}
=
\frac{1}{\prod_{q=N-m+1}^{N}(1-qx)},
\]
where $m^{\overline{i}}=m(m+1)\cdots(m+i-1)$ is the rising factorial.
\end{lemma}

\begin{proof}
We prove this identity by induction on $n = N-m \geq 0$. For the base case $n=0$ (i.e., $N=m$), both sides simplify to $\frac{1}{\prod_{q=1}^{m}(1-qx)}$, so the statement holds. 

Assume the identity holds for $n$ and any $m \geq 1$. Let
\[
F_m(n, x) = \sum_{i=0}^{n} \binom{n}{i} m^{\overline{i}} \frac{x^i}{\prod_{q=1}^{m+i}(1-qx)}.
\]
To prove it for $n+1$, we use Pascal's formula $\binom{n+1}{i} = \binom{n}{i} + \binom{n}{i-1}$ to write:
\[
F_m(n+1, x) = \sum_{i=0}^{n} \binom{n}{i} m^{\overline{i}} \frac{x^i}{\prod_{q=1}^{m+i}(1-qx)} + \sum_{i=1}^{n+1} \binom{n}{i-1} m^{\overline{i}} \frac{x^i}{\prod_{q=1}^{m+i}(1-qx)}.
\]
The first sum is $F_m(n, x)$. In the second sum, we shift the summation index to $j = i-1$:
\[
\sum_{j=0}^{n} \binom{n}{j} m^{\overline{j+1}} \frac{x^{j+1}}{\prod_{q=1}^{m+j+1}(1-qx)} = \sum_{j=0}^{n} \binom{n}{j} m^{\overline{j}} \frac{(m+j)x \cdot x^j}{\prod_{q=1}^{m+j+1}(1-qx)}.
\]
Using the algebraic decomposition
\[
\frac{(m+j)x}{\prod_{q=1}^{m+j+1}(1-qx)} = \frac{1-x}{\prod_{q=1}^{m+j+1}(1-qx)} - \frac{1}{\prod_{q=1}^{m+j}(1-qx)},
\]
the second sum becomes:
\[
(1-x) \sum_{j=0}^{n} \binom{n}{j} m^{\overline{j}} \frac{x^j}{\prod_{q=1}^{m+j+1}(1-qx)} - F_m(n, x).
\]
Adding the first sum $F_m(n, x)$, the $F_m(n, x)$ terms cancel out:
\[
F_m(n+1, x) = (1-x) \sum_{j=0}^{n} \binom{n}{j} m^{\overline{j}} \frac{x^j}{\prod_{q=1}^{m+j+1}(1-qx)} = \sum_{j=0}^{n} \binom{n}{j} m^{\overline{j}} \frac{x^j}{\prod_{q=2}^{m+j+1}(1-qx)}.
\]
Let $y = \frac{x}{1-x}$. Then for $q \geq 2$, we have $1-qx = (1-x)(1-(q-1)y)$, which yields:
\[
\prod_{q=2}^{m+j+1}(1-qx) = (1-x)^{m+j} \prod_{p=1}^{m+j}(1-py).
\]
Substituting this into the expression for $F_m(n+1, x)$, we obtain:
\[
F_m(n+1, x) = \sum_{j=0}^{n} \binom{n}{j} m^{\overline{j}} \frac{x^j}{(1-x)^{m+j} \prod_{p=1}^{m+j}(1-py)} = \frac{1}{(1-x)^m} F_m(n, y).
\]
By the induction hypothesis, $F_m(n, y) = \frac{1}{\prod_{p=n+1}^{m+n}(1-py)}$. Since $1-py = \frac{1-(p+1)x}{1-x}$, we have:
\[
F_m(n, y) = \frac{(1-x)^m}{\prod_{p=n+1}^{m+n}(1-(p+1)x)} = \frac{(1-x)^m}{\prod_{q=n+2}^{m+n+1}(1-qx)}.
\]
Therefore,
\[
F_m(n+1, x) = \frac{1}{(1-x)^m} \frac{(1-x)^m}{\prod_{q=n+2}^{m+n+1}(1-qx)} = \frac{1}{\prod_{q=n+2}^{m+n+1}(1-qx)},
\]
which completes the induction step. This concludes the proof of the lemma.    
\end{proof}

\begin{lemma}\label{lema_pomocna_dzamic}
For $1\leq m\leq N$ and $t\geq -1$ define
\[
A_m^N(t)=
\frac{1}{(m-1)!}\sum_{i=0}^{t+1}(m+i-1)!\binom{N-m}{i}\stirlingii{m+t+1}{m+i}.
\]
Then, it holds that
\[
\sum_{t\geq -1} A_m^N(t)x^{t+1}
=
\prod_{q=N-m+1}^{N}\frac{1}{1-qx}.
\]
\end{lemma}

\begin{proof}
   We use the standard generating function for Stirling numbers of the second kind \citep[see, e.g.,][Eq.~2]{kesidis2018generating}:
   \begin{align*}
       \sum_{t\geq -1}\stirlingii{m+t+1}{m+i}x^{t+1}
&= x^{-m}\sum_{t\geq -1}\stirlingii{m+t+1}{m+i}x^{m+t+1} \\
&= x^{-m}\frac{x^{m+i}}{(1-x)(1-2x)\cdots(1-(m+i)x)} \\
&= \frac{x^i}{(1-x)(1-2x)\cdots(1-(m+i)x)}
   \end{align*}

It follows that
\begin{align*}
\sum_{t\geq -1} A_m^N(t)x^{t+1}
&=\sum_{i=0}^{N-m}
\frac{(m+i-1)!}{(m-1)!}\binom{N-m}{i}
\frac{x^{i}}{\prod_{q=1}^{m+i}(1-qx)}. 
\end{align*}

Since
\[
\frac{(m+i-1)!}{(m-1)!}=m(m+1)\cdots(m+i-1),
\]
previous expression becomes
\[
\sum_{i=0}^{N-m}
\binom{N-m}{i}m^{\overline{i}}
\frac{x^{i}}{\prod_{q=1}^{m+i}(1-qx)}.
\]
The proof is concluded by applying Lemma \ref{lema_pomocna_za_pomocnu_dzamic}.

\end{proof}

\begin{lemma}\label{lema_orig_teska}
    For $1 \leq j < k \leq n$ and any $r \geq 1$ it holds that
    \begin{align}\label{jednakost_teska}
        &\frac{1}{(k-1)!}\sum_{i=0}^{r+1} (i+k-1)! \binom{n-k}{i} \stirlingii{k+r+1}{i+k} \nonumber \\
        & \qquad = \frac{1}{(k-j-1)!} \sum_{i=0}^{r+1} (k-j-1+i)! \binom{n-k+j}{i} \sum_{l=k-2}^{k+r-1} \stirlingii{l-j+2}{k-j+i} \nonumber \\
        & \qquad \qquad \qquad \qquad \qquad \times \sum_{s=0}^{k+r-l-1} \binom{n-k}{s} \frac{(s+j-1)!}{(j-1)!} \stirlingii{j+k+r-l-1}{s+j}
    \end{align}
\end{lemma}

\begin{proof}
    We apply Lemma \ref{lema_pomocna_dzamic} on both sides of the identity \eqref{jednakost_teska}.

The left-hand side (LHS) is exactly $A_k^n(r)$, since
\[
A_k^n(r)=
\frac{1}{(k-1)!}\sum_{i=0}^{r+1}(i+k-1)!\binom{n-k}{i}\stirlingii{k+r+1}{i+k}.
\]
By Lemma \ref{lema_pomocna_dzamic}, the LHS is the coefficient of $x^{r+1}$ 
in the function
\[
\prod_{q=n-k+1}^{n}\frac{1}{1-qx}.
\]

Let us move on to the right-hand side (RHS). Let
\[
h=k-j.
\]
Then $h\geq 1$. The first factor on the right-hand side can be recognized as
\[
A_{k-j}^{n}(l-k+1),
\]
since
\[
(k-j)+(l-k+1)+1=l-j+2.
\]
Indeed,
\[
A_{k-j}^{n}(l-k+1)
=\frac{1}{(k-j-1)!}\sum_{i=0}^{l-k+2}
(k-j+i-1)!\binom{n-k+j}{i}\stirlingii{l-j+2}{k-j+i}.
\]
In the given sum, the upper bound for $i$ is $r+1$, but for the considered 
range $l\leq k+r-1$, the terms outside the natural range do not contribute 
anyway, as the corresponding Stirling numbers vanish.

The second factor on the RHS is
\[
A_j^{\,n-k+j}(k+r-l-2),
\]
since
\[
j+(k+r-l-2)+1=j+k+r-l-1.
\]
Thus, the RHS can be written as
\[
\sum_{l=k-2}^{k+r-1}
A_{k-j}^{n}(l-k+1)\,A_j^{\,n-k+j}(k+r-l-2).
\]

Let us introduce the substitution
\[
u=l-k+1.
\]
As $l$ ranges from $k-2$ to $k+r-1$, $u$ ranges from $-1$ to $r$. Therefore, 
the right-hand side becomes
\[
\sum_{u=-1}^{r}
A_{k-j}^{n}(u)\,A_j^{\,n-k+j}(r-1-u).
\]
This is the convolution of the progressions $A_{k-j}^{n}$ and $A_j^{\,n-k+j}$. 
Consequently, by the definition of the Cauchy product (convolution) of two power series, the coefficient of $x^{r+1}$ in the product of their generating functions
\[
\left(\sum_{u\geq -1} A_{k-j}^{n}(u)x^{u+1}\right)\left(\sum_{v\geq -1} A_j^{\,n-k+j}(v)x^{v+1}\right)
\]
is given by
\[
\sum_{\substack{(u+1)+(v+1) = r+1 \\ u \geq -1, \, v \geq -1}} A_{k-j}^{n}(u) A_j^{\,n-k+j}(v)
= \sum_{u=-1}^{r} A_{k-j}^{n}(u) A_j^{\,n-k+j}(r-1-u),
\]
which is exactly the RHS expression.

By the Lemma \ref{lema_pomocna_dzamic}, we have
\[
\sum_{u\geq -1} A_{k-j}^{n}(u)x^{u+1}
=
\prod_{q=n-k+j+1}^{n}\frac{1}{1-qx},
\]
and also
\[
\sum_{v\geq -1} A_j^{\,n-k+j}(v)x^{v+1}
=
\prod_{q=n-k+1}^{n-k+j}\frac{1}{1-qx}.
\]
Their product is
\[
\left(\prod_{q=n-k+j+1}^{n}\frac{1}{1-qx}\right)
\left(\prod_{q=n-k+1}^{n-k+j}\frac{1}{1-qx}\right)
=
\prod_{q=n-k+1}^{n}\frac{1}{1-qx}.
\]
But this is precisely the generating function whose coefficient of $x^{r+1}$ 
is equal to $A_k^n(r)$, i.e., the left-hand side of the required identity. Thus, the LHS and the RHS have the same coefficient of 
$x^{r+1}$ in the same generating function. Therefore, they are equal
and the identity \eqref{jednakost_teska} is proven.
\end{proof}

{
\begin{lemma}\label{lem:C1_positive}
Assume that $j,k,n,r$ are integers, $1\le j<k$, $n-k>k-j$, and $r\ge 0$. Let
\[
\eta=\sum_{u=1}^{j}(n-k+u)^{r+1}.
\]
Define
\[
C_1=\frac{1}{(n-k)!}\left[
\binom{k-j+r}{r+1}
+\binom{k-j+r}{r+2}
+\eta
-\binom{k+r-1}{r+1}
-\binom{k+r-1}{r+2}
-\binom{k+r}{r+1}
\right],
\]
where $\binom{N}{M}=0$ whenever $M>N$. Then $C_1>0$.
\end{lemma}

\begin{proof}
Set
\[
a=k-j,\qquad m=n-k.
\]
Then $a\ge1$, $j\ge1$, $k=a+j$, and the assumption $n-k>k-j$ gives $m>a$. Since $m$ and $a$ are integers,
\[
m\ge a+1.
\]
Put
\[
\widetilde C_1=(n-k)!C_1=m!C_1.
\]
From the definition of $C_1$,
\begin{align*}
\widetilde C_1
&=
\binom{a+r}{r+1}
+\binom{a+r}{r+2}
+\sum_{u=1}^{j}(m+u)^{r+1} \\
&\quad
-\binom{a+j+r-1}{r+1}
-\binom{a+j+r-1}{r+2}
-\binom{a+j+r}{r+1}.
\end{align*}

By Pascal's identity,
\[
\binom{a+r}{r+1}+\binom{a+r}{r+2}
=
\binom{a+r+1}{r+2}.
\]
Also,
\begin{align*}
&\binom{a+j+r-1}{r+1}
+\binom{a+j+r-1}{r+2}
+\binom{a+j+r}{r+1} \\
&\qquad=
\binom{a+j+r}{r+2}
+\binom{a+j+r}{r+1}
=
\binom{a+j+r+1}{r+2}.
\end{align*}
Therefore,
\[
\widetilde C_1
=
\sum_{u=1}^{j}(m+u)^{r+1}
+
\binom{a+r+1}{r+2}
-
\binom{a+j+r+1}{r+2}.
\]

We now use the finite-difference identity
\[
\binom{N}{M}-\binom{L}{M}
=
\sum_{q=L}^{N-1}\binom{q}{M-1},
\]
valid for integers $N\ge L$ and $M\ge1$. This is the standard hockey-stick/finite-difference identity \citep[e.g.][Sec. 5.1]{GrahamKnuthPatashnik1994}. For completeness, we re-derive it here. By Pascal's identity,
\[
\binom{q+1}{M}
=
\binom{q}{M}
+
\binom{q}{M-1},
\]
so
\[
\binom{q+1}{M}-\binom{q}{M}
=
\binom{q}{M-1}.
\]
Summing from $q=L$ to $q=N-1$ gives
\[
\sum_{q=L}^{N-1}
\left(
\binom{q+1}{M}-\binom{q}{M}
\right)
=
\sum_{q=L}^{N-1}
\binom{q}{M-1}.
\]
The left-hand side telescopes:
\[
\sum_{q=L}^{N-1}
\left(
\binom{q+1}{M}-\binom{q}{M}
\right)
=
\binom{N}{M}-\binom{L}{M}.
\]
Hence
\[
\binom{N}{M}-\binom{L}{M}
=
\sum_{q=L}^{N-1}\binom{q}{M-1}.
\]

Apply this with
\[
N=a+j+r+1,\qquad L=a+r+1,\qquad M=r+2.
\]
Then
\begin{align*}
\binom{a+j+r+1}{r+2}
-
\binom{a+r+1}{r+2}
&=
\sum_{q=a+r+1}^{a+j+r}
\binom{q}{r+1} \\
&=
\sum_{p=0}^{j-1}
\binom{a+r+1+p}{r+1} \\
&=
\sum_{u=1}^{j}
\binom{a+r+u}{r+1}.
\end{align*}
Thus
\[
\widetilde C_1
=
\sum_{u=1}^{j}
\left(
(m+u)^{r+1}
-
\binom{a+r+u}{r+1}
\right).
\]

Now fix $u\in\{1,\dots,j\}$ and set
\[
b=a+u+1.
\]
Since $a\ge1$ and $u\ge1$, we have $b\ge3$, in particular $b\ge2$. Moreover, $m\ge a+1$ implies
\[
m+u\ge a+1+u=b.
\]
Also,
\[
\binom{a+r+u}{r+1}
=
\binom{b+r-1}{r+1}.
\]
We claim that
\[
b^{r+1}>\binom{b+r-1}{r+1}.
\]

If $r=0$, then
\[
\binom{b-1}{1}=b-1<b=b^{1}.
\]

If $r\ge1$, then
\[
\binom{b+r-1}{r+1}
=
\frac{(b-1)b(b+1)\cdots(b+r-1)}{(r+1)!}.
\]
Since $b-1<b$ and, for $t=1,\dots,r-1$,
\[
b+t\le (t+1)b,
\]
we obtain
\[
(b-1)b(b+1)\cdots(b+r-1)
<
b\cdot b\prod_{t=1}^{r-1}((t+1)b)
=
b^{r+1}r!,
\]
where the product over $t=1,\dots,r-1$ is interpreted as $1$ when $r=1$. Therefore
\[
\binom{b+r-1}{r+1}
<
\frac{b^{r+1}r!}{(r+1)!}
=
\frac{b^{r+1}}{r+1}
<
b^{r+1}.
\]
Consequently,
\[
(m+u)^{r+1}
\ge b^{r+1}
>
\binom{b+r-1}{r+1}
=
\binom{a+r+u}{r+1}.
\]
Thus every summand in
\[
\widetilde C_1
=
\sum_{u=1}^{j}
\left(
(m+u)^{r+1}
-
\binom{a+r+u}{r+1}
\right)
\]
is strictly positive. Hence
\[
\widetilde C_1>0.
\]
Since $(n-k)! = m! >0$, we conclude that
\[
C_1>0.
\]
\end{proof}
}

\subsection{Proof of the main statement}\label{main_statement}

With the technical lemmas in place, we are now ready to present and prove the main result of this work.
\begin{proof}
Let $D_t$, $t = 1, 2, \dots, j$ be the probability density function of a random variable $\sum_{i=1}^t \frac{1}{n-k+i}X_{0i}$. Then we can write
\begin{align*}
D_j(x) &= \int_0^x (n-k+j)f((n-k+j)(x-y))D_{j-1}(y)dy, \\
D_1(x) &= (n-k+1)f((n-k+1)x).
\end{align*}
By mathematical induction one can conclude that
\begin{align*}
D_{j}^{(m)}(x) &= \sum_{l=0}^{m-1}(n-k+j)^{l+1}  f^{(l)}(0)D_{j-1}^{(m-1-l)}(x) \\
& + \int_0^x (n-k+j)^{m+1} f^{(m)}((n-k+j)(x-y))D_{j-1}(y) dy,
\end{align*}
from which it follows that 
\begin{equation}\label{Dj_m} 
D_j^{(m)}(0) =  \sum_{l=0}^{m-1}(n-k+j)^{l+1}  f^{(l)}(0)D_{j-1}^{(m-1-l)}(0).
\end{equation}
{ Suppose, for a moment, $j>1$ and} put $m=j-1$. We obtain:
\[ D_j^{(j-1)}(0) = \sum_{l=0}^{j-2} (n-k+j)^{l+1} f^{(l)}(0)D_{j-1}^{(j-2-l)}(0).  \]
For a term in this sum to be non-zero, it is necessary that $j-2-l \geq j-2$, i.e. that $l=0$, from where we get
\[ D_j^{(j-1)}(0) = (n-k+j)f(0) D_{j-1}^{(j-2)}(0), \]
and, inductively, 
\begin{equation}\label{Dj_minus1}
D_j^{(j-1)}(0) = \Pi_j f^j(0),
\end{equation}
where 
\begin{align}\label{Pi_j}
    \Pi_j = \prod_{u=1}^{j} (n-k+u).
\end{align}

Now let us put $m=j$ in \eqref{Dj_m}. We have that
\[ D_j^{(j)}(0) =  \sum_{l=0}^{j-1}(n-k+j)^{l+1}  f^{(l)}(0)D_{j-1}^{(j-1-l)}(0). \]
For a term to be non-zero in this sum, it is necessary that $j-1-l \geq j-2$, i.e. $l \in \{ 0, 1 \}$, so we can conclude that
\begin{align*}
D_j^{(j)}(0) &= (n-k+j)f(0)D_{j-1}^{(j-1)}(0) + (n-k+j)^2f'(0) D_{j-1}^{(j-2)}(0) \\
&= (n-k+j)f(0)D_{j-1}^{(j-1)}(0) + (n-k+j)^2f'(0) \Pi_{j-1} f^{j-1}(0) \\
&= (n-k+j)f(0)D_{j-1}^{(j-1)}(0) + (n-k+j)f'(0) \Pi_{j} f^{j-1}(0).
\end{align*}
Inductively, we obtain:
\begin{equation}\label{D_j_j}
D_j^{(j)}(0) = \Pi_j \left( j(n-k) + \frac{j(j+1)}{2}  \right)f'(0)f^{j-1}(0),
\end{equation}
{which obviously holds also for $j=1$.}
Now, let $A_m(x) = F^{m}(x)f(x)$. For $A_m^{(u)}(0) \neq 0$ to hold, $u \geq m$ obviously needs to hold. Furthermore, we can easily calculate that 
\begin{equation}\label{A_m_m}
A_m^{(m)} (0) = m! f^{m+1}(0),
\end{equation}
and
\begin{equation}\label{A_m_plus1} 
A_m^{(m+1)}(0) =  (m+1)! \left(1 + \frac{m}{2}\right) f^m(0)f'(0),
\end{equation}
by using Leibniz' general differentiation rule. 
Now, we equate the respective densities of terms from the statement of the Theorem. After canceling $n!$ from both sides, we obtain:
\begin{align*}
\frac{1}{(k-j-1)!(n-k+j)!} & \int_0^x F^{k-j-1}(x-y)(1 - F(x-y))^{n-k+j}{f(x-y)}D_j(y)dy \\
&= \frac{1}{(k-1)!(n-k)!} F^{k-1}(1 - F(x))^{n-k}f(x).
\end{align*}
After applying binomial formula to $(1 - F(x-y))^{n-k+j}$ and $(1 - F(x))^{n-k}$, rearranging the terms and writing $F^{k-1+i}(x) ={(k-1+i)}\int_0^x A_{k-2+i}(y)dy$, we get:
\begin{align}\label{jednakost_gustina}
&\frac{1}{(k-j-1)!(n-k+j)!}  \sum_{i=0}^{n-k+j} (-1)^i \binom{n-k+j}{i} \int_0^x A_{k-j-1+i}(x-y)D_j(y) dy \nonumber \\
& \;\;\;\;\;\; = \frac{1}{(k-1)!(n-k)!} \sum_{i=0}^{n-k} (-1)^i \binom{n-k}{i}(k-1+i)f(x) \int_0^x A_{k-2+i}(y)dy.
\end{align}
After differentiating both sides $k$ times we obtain: 
\begin{align}\label{diferencirano_k_puta}
& \frac{1}{(k-j-1)!(n-k+j)!}  \sum_{i=0}^{n-k+j} (-1)^i \binom{n-k+j}{i} \left[ \sum_{l=0}^{k-1}D_j^{(k-1-l)}(x)A_{k-j-1+i}^{(l)}(0) \right. \nonumber \\
& \left. \qquad \qquad \qquad \qquad \qquad \qquad \qquad \qquad \qquad \qquad \qquad \qquad  + \int_0^x A_{k-j-1+i}^{(k)}(x-y)D_j(y) dy \right] \nonumber \\
& = \frac{1}{(k-1)!(n-k)!} \sum_{i=0}^{n-k} (-1)^i \binom{n-k}{i}(k-1+i) \left[ \sum_{l=1}^k {\binom{k}{l}} f^{(k-l)}(x)A_{i+k-2}^{(l-1)}(x) \right. \nonumber \\
&  \left. \qquad \qquad \qquad \qquad \qquad \qquad \qquad \qquad  \qquad \qquad \qquad \qquad + f^{(k)}(x) \int_0^x A_{i+k-2}(y)dy \right].
\end{align}
By setting $x=0$ and canceling the integrals (for $x=0$ they are equal to zero), we get:
\begin{align}\label{diferencirano_u_nuli}
& \frac{1}{(k-j-1)!(n-k+j)!}  \sum_{i=0}^{n-k+j} (-1)^i \binom{n-k+j}{i} \sum_{l=0}^{k-1}D_j^{(k-1-l)}(0)A_{k-j-1+i}^{(l)}(0) \nonumber \\
& \;\;\; = \frac{1}{(k-1)!(n-k)!} \sum_{i=0}^{n-k} (-1)^i \binom{n-k}{i}(k-1+i) \sum_{l=1}^k \binom{k}{l} f^{(k-l)}(0)A_{i+k-2}^{(l-1)}(0).
\end{align} 
For $A_{i+k-2}^{(l-1)}(0)$ to be nonzero, $l-1 \geq i+k-2$, i.e. $l \geq i+k-1$ needs to hold, which is achieved in only three cases:
\begin{enumerate}
\item $i=0$ and $l = k-1$;
\item $i=0$ and $l=k$;
\item $i=1$ and $l = k$.
\end{enumerate}
Let $R_{0, k-1}, R_{0, k}$ and $R_{1, k}$ be the corresponding terms on the right-hand side of \eqref{diferencirano_u_nuli}, including scaling constants. Now we compute them. First we have:
\begin{align*}
R_{0, k-1} &= \frac{1}{(k-1)!(n-k)!}(k-1)k f'(0) A_{k-2}^{(k-2)}(0) \\
&= \frac{1}{(k-2)!(n-k)!}(k-2)!k f'(0)f^{k-1}(0) \\
&= \frac{1}{(n-k)!}k f'(0)f^{k-1}(0).
\end{align*}
Similarly, using \eqref{A_m_plus1} we can obtain:
\begin{align*}
R_{0, k} &= \frac{1}{(k-1)!(n-k)!} (k-1) f(0) A_{k-2}^{(k-1)}(0) \\
&= \frac{k-1}{(k-1)!(n-k)!} f(0)(k-1)! \left(  f^{k-2}(0)f'(0) + f'(0) f^{k-2}(0)\frac{k-2}{2}  \right) \\
&= \frac{k-1}{(n-k)!}f'(0)f^{k-1}(0) \left( 1 + \frac{k-2}{2} \right) \\
&= \frac{1}{(n-k)!}\frac{k(k-1)}{2} f'(0)f^{k-1}(0).
\end{align*}
Finally, we have:
\begin{align*}
R_{1, k} &= -\frac{1}{(k-1)!(n-k)!} k(n-k) f(0) A_{k-1}^{(k-1)}(0) \\
&= -\frac{1}{(k-1)!(n-k)!} k(n-k) f(0) (k-1)! f^k(0) \\
&= -\frac{k(n-k)}{(n-k)!} f^{k+1}(0).
\end{align*}

Now we focus on the left-hand side of \eqref{diferencirano_u_nuli}. For  $D_j^{(k-1-l)}(0) A_{k-j-1+i}^{(l)}(0)$ to be nonzero, $k-1-l \geq j-1$ and $l \geq k-j-1+i$ need to hold. Combining these two, we get the condition
\begin{equation}\label{uslov_za_l}
k-j-1 +i \leq l \leq k-j,
\end{equation}
which holds in exactly three cases:
\begin{enumerate}
\item $i=0$ and $l = k-j-1$;
\item $i=0$ and $l = k-j$;
\item $i=1$ and $l= k-j$.
\end{enumerate}
Corresponding terms on the left side (including scaling constants) will be denoted as $L_{0,k-j-1}$, $L_{0, k-j}$ and $L_{1, k-j}$, respectively. We have:
\begin{align*}
L_{0, k-j-1} &= \frac{1}{(k-j-1)!(n-k+j)!} D_{j}^{(k-1-(k-j-1))}(0) A_{k-j-1+0}^{(k-j-1)}(0) \\
&= \frac{1}{(k-j-1)!(n-k+j)!} D_j^{(j)}(0) A_{k-j-1}^{(k-j-1)}(0) \\
&= \frac{1}{(n-k+j)!}\Pi_j \left( j(n-k) + \frac{j(j+1)}{2} \right)f'(0)f^{j-1}(0)f^{k-j}(0) \\
&= \frac{\Pi_j}{(n-k+j)!} \left( j(n-k) + \frac{j(j+1)}{2} \right)f'(0)f^{k-1}(0) \\
&= \frac{1}{(n-k)!} \left( j(n-k) + \frac{j(j+1)}{2} \right)f'(0)f^{k-1}(0) \\
&= \frac{1}{(n-k)!} \frac{2nj - 2kj + j^2 + j}{2}  f'(0)f^{k-1}(0).
\end{align*}
Similarly, 
\begin{align*}
L_{0, k-j} &= \frac{1}{(k-j-1)!(n-k+j)!} D_{j}^{(j-1)}(0) A_{k-j-1}^{(k-j)}(0) \\
&= \frac{\Pi_j}{(k-j-1)!(n-k+j)!} f^j(0) \cdot  (k-j)! \left(1+ \frac{k-j-1}{2}  \right)f'(0)f^{k-1-j}(0) \\
&=\frac{k-j}{(n-k)!}\frac{k-j+1}{2}f'(0) f^{k-1}(0) \\
&= \frac{1}{(n-k)!} \frac{k^2 -2kj + j^2 +k -j}{2} f'(0) f^{k-1}(0).
\end{align*}
Finally, we have: 
\begin{align*}
L_{1, k-j} &= - \frac{1}{(k-j-1)!(n-k+j)!}(n-k+j) D_j^{(k-1-(k-j))}(0) A_{k-j-1+1}^{(k-j)}(0) \\
&=- \frac{1}{(k-j-1)!(n-k+j)!}(n-k+j) D_j^{(j-1)} A_{k-j}^{(k-j)} (0) \\
&= - \frac{\Pi_j}{(k-j-1)!(n-k+j)!}(n-k+j) f^j(0) (k-j)! f^{k-j+1}(0) \\
&= -\frac{(k-j)(n-k+j)}{(n-k)!}f^{k+1}(0).
\end{align*}
Equation \eqref{diferencirano_u_nuli} now becomes:
\[  L_{0, k-j-1} + L_{0, k-j} + L_{1, k-j} = R_{0, k-1} +  R_{0, k} + R_{1, k},  \]
which we rearrange to obtain: 
\begin{equation}\label{rearrangement} 
-R_{1, k} + L_{1, k-j} = R_{0, k-1} +  R_{0, k} - (L_{0, k-j-1} + L_{0, k-j}). 
\end{equation}
The rearrangement is conducted because  $L_{1, k-j}$ and $R_{1, k}$ have $f^{k+1}(0)$ as factors, and others have $f^{k-1}(0)$. Now, we calculate three sums that appear in \eqref{rearrangement}. First: 
\begin{align*}
-R_{1, k} + L_{1, k-j} &= \frac{k(n-k)}{(n-k)!} f^{k+1}(0) -\frac{(k-j)(n-k+j)}{(n-k)!}f^{k+1}(0) \\
&= \frac{1}{(n-k)!} \left(  j^2 -2kj +nj \right)f^{k+1}(0).
\end{align*}
Next, 
\begin{align*}
R_{0, k-1} + R_{0, k} &= \frac{1}{(n-k)!}k f'(0)f^{k-1}(0)  + \frac{1}{(n-k)!}\frac{k(k-1)}{2} f'(0)f^{k-1}(0) \\
&= \frac{1}{(n-k)!} \left( k + \frac{k(k-1)}{2}  \right) f'(0)f^{k-1}(0) \\
&= \frac{1}{(n-k)!} \frac{k^2 + k}{2} f'(0)f^{k-1}(0).
\end{align*}
Finally:
\begin{align*}
L_{0, k-j-1} + L_{0, k-j} &= \frac{1}{(n-k)!} \frac{2nj - 2kj + j^2 + j}{2}  f'(0)f^{k-1}(0) \\
& \;\;\;\;\;\;\;\;\;\;\;\;\;\;\;\;\;\;\;\; + \frac{1}{(n-k)!} \frac{k^2 -2kj + j^2 +k -j}{2} f'(0) f^{k-1}(0) \\
&= \frac{1}{(n-k)!} \frac{2nj - 4kj + 2j^2 + k + k^2}{2}f'(0)f^{k-1}(0).
\end{align*}
Now, we have that
\begin{align*}
R_{0, k-1} + R_{0, k} &- (L_{0, k-j-1} + L_{0, k-j})\\
&= \frac{1}{(n-k)!} (-nj +2kj -j^2)   f'(0)f^{k-1}(0) \\
&= -\frac{1}{(n-k)!} (j^2 -2kj +nj) f'(0)f^{k-1}(0).
\end{align*}
If we go back to \eqref{rearrangement}, we have 
\begin{equation}\label{problematicna}
\frac{1}{(n-k)!} (j^2 -2kj +nj) f^{k+1}(0) = -\frac{1}{(n-k)!} \left(  j^2 -2kj +nj \right)f'(0)f^{k-1}(0). 
\end{equation}

{Since $j \neq 2k - n$, dividing both sides by $-\frac{1}{(n-k)!}\left( j^2 - 2kj + nj \right) f^{k-1}(0) \neq 0$ directly yields
\[
f'(0) = -f^2(0),
\]
which establishes \eqref{uslov_q} for the base case $q=1$. }

Suppose that \eqref{uslov_q} holds for all $q \leq r$. We have to prove that it holds for $q=r+1$. Let us differentiate equation \eqref{jednakost_gustina} $k+r$ times and substitute $x=0$. We obtain:
\begin{align}\label{diferencirano_k_plus_r_puta_u_nuli}
&\frac{1}{(k-j-1)!(n-k+j)!}  \sum_{i=0}^{n-k+j} (-1)^i \binom{n-k+j}{i}  \sum_{l=0}^{k+r-1} D_j^{(k+r-1-l)}(0)A_{k-j-1+i}^{(l)}(0)   \nonumber \\
&= \frac{1}{(k-1)!(n-k)!} \sum_{i=0}^{n-k} (-1)^i \binom{n-k}{i} (k-1+i)  \sum_{l=1}^{k+r} \binom{k+r}{l} f^{(k+r-l)}(0)A_{i+k-2}^{(l-1)}(0). 
\end{align}
Analogously to earlier discussions, we cancel zero terms and get:
\begin{align}
&\frac{1}{(k-j-1)!(n-k+j)!}  \sum_{i=0}^{r+1} (-1)^i \binom{n-k+j}{i}  \sum_{l=k-j-1+i}^{k+r-j} D_j^{(k+r-1-l)}(0)A_{k-j-1+i}^{(l)}(0)   \nonumber \\
&= \frac{1}{(k-1)!(n-k)!} \sum_{i=0}^{r+1} (-1)^i \binom{n-k}{i} (k-1+i)  \sum_{l=i+k-1}^{k+r} \binom{k+r}{l} f^{(k+r-l)}(0)A_{i+k-2}^{(l-1)}(0). \label{priprema_za_ind}
\end{align}
To apply the induction hypothesis on $A_{u}^{(v)} (0)$, $v-u \leq r$ has to hold, which for our left-hand side means that $l - (k-j-1 + i) \leq r$ needs to hold for every $l$. It does, except for $i=0$ where it is not satisfied for $l = k+r-j$. The same applies to  the right-hand side. Because of that, we split each side of the equation \eqref{priprema_za_ind} into two: one for $i=0$ and one for $i>0$. It becomes:
\begin{align*}
&\frac{1}{(k-j-1)!(n-k+j)!}   \sum_{l=k-1}^{k+r} D_j^{(k+r-1-l+j)}(0)A_{k-j-1}^{(l-j)}(0)   \nonumber \\
& + \frac{1}{(k-j-1)!(n-k+j)!}  \sum_{i=1}^{r+1} (-1)^i \binom{n-k+j}{i}  \sum_{l=i+k-1}^{k+r} D_j^{(k+r-1-l+j)}(0)A_{k-j-1+i}^{(l-j)}(0)   \nonumber \\
&= \frac{1}{(k-1)!(n-k)!}  (k-1)  \sum_{l=k-1}^{k+r} \binom{k+r}{l} f^{(k+r-l)}(0)A_{k-2}^{(l-1)}(0) \nonumber \\
& + \frac{1}{(k-1)!(n-k)!} \sum_{i=1}^{r+1} (-1)^i \binom{n-k}{i} (k-1+i)  \sum_{l=i+k-1}^{k+r} \binom{k+r}{l} f^{(k+r-l)}(0)A_{i+k-2}^{(l-1)}(0). 
\end{align*}
In the first sum, as well in the inner sum in the second row, $l$ is shifted up by $j$, to obtain the same counter as in the inner sum in the last row. Let us write the upper equality as $L_0 + L_+ = R_0 + R_+$, following the order of the terms respectively. 

First we discuss $L_0$. Induction hypothesis (using Lemma \eqref{A_m_izraz}) can be applied to $A_{k-j-1}^{(l-1)}$ for every $l$ except the $k+r$. Because of that, we write:
\begin{align*}
L_0 &= \frac{1}{(k-j-1)!(n-k+j)!} D_j^{(j-1)}(0) A_{k-j-1}^{(k+r-j)}(0) \nonumber \\
& + \frac{1}{(k-j-1)!(n-k+j)!}   \sum_{l=k-1}^{k+r-1} D_j^{(k+r-1-l+j)}(0)A_{k-j-1}^{(l-j)}(0)   \nonumber \\
&= L_0^1 + L_0^2,
\end{align*}
where $L_0^1$ and $L_0^2$ are defined to be two terms in the upper sum, respectively. We have that 
\begin{align}
L_0^1 &= \frac{1}{(k-j-1)!(n-k+j)!} \Pi_j f^j(0) A_{k-j-1}^{(k-j+r)}{(0)} \nonumber \\
&= \frac{f^j(0)}{(k-j-1)!(n-k)!}A_{k-j-1}^{(k-j+r)}{(0)}. \label{L_0^1}
\end{align}
By discussing values of $l_1, l_2, \dots l_{k-j}$ we can obtain that 
\begin{align*}
A_{k-j-1}^{(k-j+r)}{(0)} &= \sum_{\substack{l_1, \dots l_{k-j -1 } \geq 1, l_{k-j} \geq 0 \\ l_1 + \cdots + l_{k-j} = k-j+r}} \binom{k-j+r}{l_1, \dots, l_{k-j}} f^{(l_1 - 1)}(0) \cdots f^{(l_{k-j-1} - 1)}(0) f^{(l_{k-j})}(0) \\
&=\sum_{\substack{1 \leq l_1, \dots l_{k-j -1 } \leq r+1, \\ 0 \leq l_{k-j} \leq r \\ l_1 + \cdots + l_{k-j} = k-j+r}}  \binom{k-j+r}{l_1, \dots, l_{k-j}} f^{(l_1 - 1)}(0) \cdots f^{(l_{k-j-1} - 1)}(0) f^{(l_{k-j})}(0) \\
&+ \frac{(k-j+r)!}{(r+1)!} f^{k-j-1}(0) f^{(r+1)}(0) \\
&+ \frac{(k-j+r)!}{(r+2)!}(k-j-1)f^{k-j-1}(0) f^{(r+1)}(0) \\
&= (-1)^{r+1}f^{k-j+r+1}(0) \sum_{\substack{1 \leq l_1, \dots l_{k-j -1 } \leq r+1, \\ 0 \leq l_{k-j} \leq r \\ l_1 + \cdots + l_{k-j} = k-j+r}}  \frac{(k-j+r)!}{l_1! \cdots l_{k-j}!} \\
&+ \frac{(k-j+r)!}{(r+1)!} f^{k-j-1}(0) f^{(r+1)}(0) \\
&+ \frac{(k-j+r)!}{(r+2)!}(k-j-1)f^{k-j-1}(0) f^{(r+1)}(0).
\end{align*}
If we go back and substitute this into the expression \eqref{L_0^1} for $L_0^1$, we get:
\begin{align*}
L_0^1 &= \frac{1}{(k-j-1)!(n-k)!}(-1)^{r+1}f^{k+r+1}(0) \sum_{\substack{1 \leq l_1, \dots l_{k-j -1 } \leq r+1, \\ 0 \leq l_{k-j} \leq r \\ l_1 + \cdots + l_{k-j} = k-j+r}}  \frac{(k-j+r)!}{l_1! \cdots l_{k-j}!} \\
&+ \frac{(k-j+r)!}{(r+1)!(k-j-1)!(n-k)!} f^{k-1}(0) f^{(r+1)}(0) \\
&+ \frac{(k-j+r)!}{(r+2)!(k-j-1)!(n-k)!}(k-j-1)f^{k-1}(0) f^{(r+1)}(0).
\end{align*}
Now, we compute $L_0^2$. By shifting the summation index and applying the induction hypothesis we obtain:
\begin{align*}
L_0^2 &= \frac{1}{(k-j-1)!(n-k+j)!} \sum_{l=k-1}^{k+r-1} D_j^{(k+r-1-l+j)}(0)A_{k-j-1}^{(l-j)}(0)\\
&= \frac{1}{(k-j-1)!(n-k+j)!} \sum_{u=0}^r D_j^{(j+r-u)}(0)A_{k-j-1}^{(k-j-1 + u)}(0) \\
&= \frac{1}{(k-j-1)!(n-k+j)!} \sum_{u=0}^r D_j^{(j+r-u)}(0) (-1)^{u} f^{k-j+u}(0) \stirlingii{k-j+u}{k-j}(k-j-1)! \\
&=\frac{1}{(n-k+j)!} \sum_{l=0}^r  (-1)^{l} f^{k-j+l}(0) \stirlingii{k-j+l}{k-j}D_j^{(j+r-l)}(0).
\end{align*}

Now, we need to compute $D_j^{(j+r-l)}(0)$. Denote 
\begin{equation}\label{pi_definicija} 
\pi_{u_1, \dots, u_j} = (n-k+j)^{u_1} (n-k+j-1)^{u_2}\cdots (n-k+j - (j-1))^{u_j}, 
\end{equation}
where $\Pi_j$ is defined in the equation \eqref{Pi_j}.
By using the recurrent relation derived for $D_j^{(m)}$ and mathematical induction, we can get
\begin{align*}
D_j^{(j+r-l)}(0) &= \Pi_j \sum_{\substack{u_1, \dots u_j \geq 0 \\ u_1 + \cdots + u_j = r-l+1}}  \pi_{u_1, \dots, u_j}  f^{(u_1)}(0) \cdots f^{(u_j)} (0). 
\end{align*}

Assume for a moment that $l>0$, because we can apply induction hypothesis without discussion. For $l>0$ we have:
\begin{align*}
D_j^{(j+r-l)}(0) &= \Pi_j (-1)^{r-l+1} f^{r+j+1-l} (0) \sum_{\substack{u_1, \dots u_j \geq 0 \\ u_1 + \cdots + u_j = r-l+1}} \pi_{u_1, \dots, u_j}  \\
\end{align*}
For $l=0$ we have that 
\begin{align*}
D_j^{(j+r)}(0) &= \Pi_j \sum_{\substack{u_1, \dots u_j \geq 0 \\ u_1 + \cdots + u_j = r+1}}  \pi_{u_1, \dots, u_j}  f^{(u_1)}(0) \cdots f^{(u_j)} (0) \\
&= \Pi_jf^{j-1}(0) f^{(r+1)}(0) \eta + \Pi_j (-1)^{r+1} f^{r+j+1}(0)\sum_{\substack{u_1, \dots u_j \geq 0 \\ u_1 + \cdots + u_j = r+1 \\ (\forall t) u_t \neq r+1}}  \pi_{u_1, \dots, u_j},
\end{align*}
where
\[ \eta =  (n-k+j)^{r+1} + \cdots + (n-k+j-(j-1))^{r+1} . \]
Having this, we can calculate that
\begin{align*}
L_0^2 &= \frac{1}{(n-k+j)!} \sum_{l=0}^r (-1)^l f^{k-j+l}(0)\stirlingii{k-j+l}{k-j}D_j^{(j+r-l)}(0) \\
&=  \frac{1}{(n-k)!}(-1)^{r+1}f^{k+r+1}(0) \sum_{l=1}^r \stirlingii{k-j+l}{k-j} \sum_{\substack{u_1, \dots, u_j \geq 0 \\ u_1 + \cdots u_j = r-l+1}} \pi_{u_1,  \dots, u_j} \\
&+ \frac{1}{(n-k)!} f^{k-1}(0) f^{(r+1)}(0) \eta \\
&+ \frac{1}{(n-k)!} (-1)^{r+1} f^{k+r+1}(0) \sum_{\substack{u_1, \dots, u_j \geq 0 \\ u_1 + \cdots u_j = r+1 \\ (\forall t) u_t \neq r+1}} \pi_{u_1, \dots, u_j}. 
\end{align*}

At this point, we have calculated $L_0$. Now, let us calculate $L_+$. By shifting $l$ down by one we obtain:
\begin{align*}
L_+ &= \frac{1}{(k-j-1)!(n-k+j)!} \sum_{i=1}^{r+1} (-1)^i \binom{n-k+j}{i} \sum_{l=i+k-2}^{k+r-1} D_j^{(j+k+r-l-2)}(0) A_{k-j-1+i}^{(l-j+1)}(0).
\end{align*}
By the induction hypothesis we can expand $A_{k-j-1+i}^{(l-j+1)}(0)$ and then we have
\begin{align*}
L_+ &= \frac{1}{(k-j-1)!(n-k+j)!} \sum_{i=1}^{r+1} (-1)^i \binom{n-k+j}{i} \cdot \\
& \cdot \sum_{l=i+k-2}^{k+r-1}  (-1)^{l+2-k-i} f^{l-j+2} (0) \stirlingii{l-j+2}{k-j+i} (k-j-1+i)!  D_j^{(j+k+r-l-2)}(0) \\
&= \frac{1}{(k-j-1)!(n-k+j)!} \sum_{i=1}^{r+1} (k-j-1+i)! (-1)^i \binom{n-k+j}{i} \cdot \\
& \cdot \sum_{l=i+k-2}^{k+r-1}  (-1)^{l+2-k-i} f^{l-j+2} (0) \stirlingii{l-j+2}{k-j+i}  D_j^{(j+k+r-l-2)}(0).
\end{align*}
Now, we compute $D_j^{(j+k+r-l-2)}(0)$. We have that:
\begin{align*}
D_j^{(j+k+r-l-2)}(0) &= \Pi_j \sum_{\substack{u_1, \dots, u_j \geq 0 \\ u_1 + \cdots u_j = k+r-l-1}} \pi_{u_1, \dots, u_j} f^{(u_1)}(0) \cdots f^{(u_j)}(0) \\
&= \Pi_j (-1)^{k+r-l-1}f^{j+k+r-l-1}(0) \sum_{\substack{u_1, \dots, u_j \geq 0 \\ u_1 + \cdots u_j = k+r-l-1}} \pi_{u_1, \dots, u_j}.
\end{align*}
Substituting this back into the upper expression for $L_+$ we obtain:
\begin{align*}
L_+ &= \frac{1}{(k-j-1)!(n-k)!} (-1)^{r+1}f^{k+r+1}(0)  \sum_{i=1}^{r+1} (k-j-1+i)! \binom{n-k+j}{i} \cdot \\
& \cdot \sum_{l=i+k-2}^{k+r-1} \stirlingii{l-j+2}{k-j+i}\sum_{\substack{u_1, \dots, u_j \geq 0 \\ u_1 + \cdots u_j = k+r-l-1}} \pi_{u_1, \dots, u_j}.
\end{align*}

Similar discussion is to be done for $R_0 + R_+$. 
{
Here, extra care is
needed: in the sum over $l$, the summand $l=k-1$ contains the factor
$f^{(k+r-l)}(0)=f^{(r+1)}(0)$, which is the derivative that has to be
determined and therefore must not be replaced by the induction
hypothesis. We isolate it first. Using $A_{k-2}^{(k-2)}(0)=(k-2)!f^{k-1}(0)$
from \eqref{A_m_m}, the term for $l=k-1$ equals
\[\frac{k-1}{(k-1)!(n-k)!}\binom{k+r}{k-1}(k-2)!\,f^{(r+1)}(0)f^{k-1}(0)
=\frac{1}{(n-k)!}\binom{k+r}{k-1}f^{(r+1)}(0)f^{k-1}(0).\]
For $k\le l\le k+r-1$ the induction hypothesis applies to both factors,
and the term $l=k+r$ is expanded as before. Hence
\begin{align*}
R_0 &= \frac{k-1}{(k-1)!(n-k)!} \sum_{l=k-1}^{k+r} \binom{k+r}{l} f^{(k+r-l)}(0) A_{k-2}^{(l-1)}(0) \\
&= \frac{1}{(n-k)!}\binom{k+r}{k-1} f^{(r+1)}(0)f^{k-1}(0) \\
&\quad + \frac{1}{(n-k)!}(-1)^{r+1} f^{k+r+1}(0) \sum_{l=k}^{k+r-1} \binom{k+r}{l} \stirlingii{l}{k-1} \\
&\quad + \frac{(k+r-1)!}{(r+1)!\,(k-2)!\,(n-k)!} f^{k-1}(0) f^{(r+1)}(0) \\
&\quad + \frac{(k+r-1)!\,(k-2)}{(r+2)!\,(k-2)!\,(n-k)!} f^{k-1}(0) f^{(r+1)}(0) \\
&\quad + \frac{1}{(k-2)!(n-k)!} (-1)^{r+1} f^{k+r+1}(0)
\sum_{\substack{1 \le l_1, \dots, l_{k-2} \le r+1 \\ 0 \le l_{k-1} \le r \\ l_1 + \cdots + l_{k-1} = k+r-1}}
\frac{(k+r-1)!}{l_1! \cdots l_{k-1}!}.
\end{align*}
}

Now, let us calculate $R_+$. By applying the induction hypothesis we obtain
\begin{align*}
R_+ &= \frac{1}{(k-1)!(n-k)!} \sum_{i=1}^{r+1} (-1)^i \binom{n-k}{i} (k-1+i) \cdot \\
&  \cdot  \sum_{l=i+k-1}^{k+r} \binom{k+r}{l} (-1)^{k+r-l} f^{k+r-l+1}(0) (-1)^{l-i-k +1} f^{l}(0) \stirlingii{l}{i+k-1} (i+k-2)! \\
&= \frac{1}{(k-1)!(n-k)!}(-1)^{r+1} f^{k+r+1}(0) \sum_{i=1}^{r+1} (i+k-1)! \binom{n-k}{i} \sum_{l=i+k-1}^{k+r} \binom{k+r}{l} \stirlingii{l}{i+k-1}.
\end{align*}

{
At this point we collect, on both sides of $L_0^1+L_0^2+L_+=R_0+R_+$, the
terms containing $f^{(r+1)}(0)f^{k-1}(0)$ and those containing
$(-1)^{r+1}f^{k+r+1}(0)$, and divide by $f^{k-1}(0)>0$. This yields the expression of the form
\[
C_1 f^{(r+1)}(0) = C_2 (-1)^{r+1} f^{r+2}(0),
\]
where, with $\eta=\sum_{u=1}^{j}(n-k+u)^{r+1}$ and the convention
$\binom{N}{M}=0$ for $M>N$,
\begin{align*}
(n-k)!\,C_1 &=
\binom{k-j+r}{k-j-1} + \binom{k-j+r}{k-j-2} + \eta
- \binom{k+r-1}{k-2} - \binom{k+r-1}{k-3} - \binom{k+r}{r+1},
\end{align*}
and $(n-k)!\,C_2$ equals the right-hand side displayed below:
\begin{align*}
(n-k)!\,C_2 &=
\frac{1}{(k-2)!} \sum_{\substack{1 \le l_1, \dots, l_{k-2} \le r+1 \\ 0 \le l_{k-1} \le r \\ l_1 + \cdots + l_{k-1} = k+r-1}} \frac{(k+r-1)!}{l_1! \cdots l_{k-1}!} \\
&\quad + \sum_{l=k}^{k+r-1} \binom{k+r}{l} \stirlingii{l}{k-1} \\
&\quad + \frac{1}{(k-1)!} \sum_{i=1}^{r+1} (i+k-1)! \binom{n-k}{i} \sum_{l=i+k-1}^{k+r} \binom{k+r}{l} \stirlingii{l}{i+k-1} \\
&\quad - \frac{1}{(k-j-1)!} \sum_{i=1}^{r+1} (k-j-1+i)! \binom{n-k+j}{i}
\sum_{l=i+k-2}^{k+r-1} \stirlingii{l-j+2}{k-j+i}
\sum_{\substack{u_1, \dots, u_j \ge 0 \\ u_1 + \cdots + u_j = k+r-l-1}} \pi_{u_1, \dots, u_j} \\
&\quad - \sum_{l=1}^r \stirlingii{k-j+l}{k-j}
\sum_{\substack{u_1, \dots, u_j \ge 0 \\ u_1 + \cdots + u_j = r-l+1}} \pi_{u_1, \dots, u_j} - \sum_{\substack{u_1, \dots, u_j \geq 0 \\ u_1 + \cdots u_j = r+1 \\ (\forall t) u_t \neq r+1}} \pi_{u_1, \dots, u_j}
\\
&- \frac{1}{(k-j-1)!} \sum_{\substack{1 \le l_1, \dots, l_{k-j-1} \le r+1 \\ 0 \le l_{k-j} \le r \\ l_1 + \cdots + l_{k-j} = k-j+r}}
\frac{(k-j+r)!}{l_1! \cdots l_{k-j}!}.
\end{align*}
}

{The condition $n-k > k-j$, i.e. $n+j > 2k$ ensures that $C_1>0$ for any $r \geq 0$ (Lemma \ref{lem:C1_positive}) so it can eventually be divided with.} If we prove that $C_1 = C_2$, then we have successfully proven the induction step. We will write $C_1 = C_2$ and show that it is equivalent to the expression that is known to be true. Since
\[ \sum_{l=k}^{k+r-1}  \binom{k+r}{l} \stirlingii{l}{k-1} + \binom{k+r}{k-1} = \sum_{l=k-1}^{k+r-1}  \binom{k+r}{l} \stirlingii{l}{k-1}, \]
 the equality $C_1 = C_2$ can be written as
\begin{align*}
&\binom{k-j+r}{k-j-1} + \binom{k-j+r}{k-j-2}  + \eta - \binom{k+r-1}{k-2} - \binom{k+r-1}{k-3}\\
&= \frac{1}{(k-2)!} \sum_{\substack{1 \leq l_1, \dots, l_{k-2} \leq r+1 \\ 0 \leq l_{k-1} \leq r \\ l_1 + \cdots + l_{k-1} = k+r-1}} \frac{(k+r-1)!}{l_1! \cdots l_{k-1}!} \\
& +  \sum_{l=k-1}^{k+r-1}  \binom{k+r}{l} \stirlingii{l}{k-1} \\
& + \frac{1}{(k-1)!} \sum_{i=1}^{r+1} (i+k-1)! \binom{n-k}{i} \sum_{l=i+k-1}^{k+r} \binom{k+r}{l} \stirlingii{l}{i+k-1} \\
& - \frac{1}{(k-j-1)!}  \sum_{i=1}^{r+1} (k-j-1+i)! \binom{n-k+j}{i} \sum_{l=i+k-2}^{k+r-1} \stirlingii{l-j+2}{k-j+i}\sum_{\substack{u_1, \dots, u_j \geq 0 \\ u_1 + \cdots u_j = k+r-l-1}} \pi_{u_1, \dots, u_j} \\
&-  \sum_{l=1}^r \stirlingii{k-j+l}{k-j} \sum_{\substack{u_1, \dots, u_j \geq 0 \\ u_1 + \cdots u_j = r-l+1}} \pi_{u_1,  \dots, u_j}  -  \sum_{\substack{u_1, \dots, u_j \geq 0 \\ u_1 + \cdots u_j = r+1 \\ (\forall t) u_t \neq r+1}} \pi_{u_1, \dots, u_j}  \\
& - \frac{1}{(k-j-1)!}\sum_{\substack{1 \leq l_1, \dots l_{k-j -1 } \leq r+1, \\ 0 \leq l_{k-j} \leq r \\ l_1 + \cdots + l_{k-j} = k-j+r}}  \frac{(k-j+r)!}{l_1! \cdots l_{k-j}!}.
\end{align*}
After rearranging the terms, we get
\begin{align*}
&\binom{k-j+r}{k-j-1} + \binom{k-j+r}{k-j-2} + \frac{1}{(k-j-1)!}\sum_{\substack{1 \leq l_1, \dots l_{k-j -1 } \leq r+1, \\ 0 \leq l_{k-j} \leq r \\ l_1 + \cdots + l_{k-j} = k-j+r}}  \frac{(k-j+r)!}{l_1! \cdots l_{k-j}!}\\
&  + \sum_{\substack{u_1, \dots, u_j \geq 0 \\ u_1 + \cdots u_j = r+1 }} \pi_{u_1, \dots, u_j}  \\
&=  \binom{k+r-1}{k-2} + \binom{k+r-1}{k-3} + \frac{1}{(k-2)!} \sum_{\substack{1 \leq l_1, \dots, l_{k-2} \leq r+1 \\ 0 \leq l_{k-1} \leq r \\ l_1 + \cdots + l_{k-1} = k+r-1}} \frac{(k+r-1)!}{l_1! \cdots l_{k-1}!} \\
& +  \sum_{l=k-1}^{k+r-1}  \binom{k+r}{l} \stirlingii{l}{k-1} \\
& + \frac{1}{(k-1)!} \sum_{i=1}^{r+1} (i+k-1)! \binom{n-k}{i} \sum_{l=i+k-1}^{k+r} \binom{k+r}{l} \stirlingii{l}{i+k-1} \\
& - \frac{1}{(k-j-1)!}  \sum_{i=1}^{r+1} (k-j-1+i)! \binom{n-k+j}{i} \sum_{l=i+k-2}^{k+r-1} \stirlingii{l-j+2}{k-j+i}\sum_{\substack{u_1, \dots, u_j \geq 0 \\ u_1 + \cdots u_j = k+r-l-1}} \pi_{u_1, \dots, u_j} \\
&-  \sum_{l=1}^r \stirlingii{k-j+l}{k-j} \sum_{\substack{u_1, \dots, u_j \geq 0 \\ u_1 + \cdots u_j = r-l+1}} \pi_{u_1,  \dots, u_j} .
\end{align*}
First binomial coefficient on the left-hand side, as well as on the right, will be combined with adequate sum to make the last index take value $r+1$. The second one will also be combined with adequate sum, because if last index is equal to zero, one of the others can take value $r+2$. Doing that, we obtain
\begin{align*}
&  \frac{1}{(k-j-1)!} \left( \sum_{\substack{1 \leq l_1, \dots l_{k-j} \\ l_1 + \cdots + l_{k-j} = k-j+r}}  \frac{(k-j+r)!}{l_1! \cdots l_{k-j}!} + \sum_{\substack{1 \leq l_1, \dots l_{k-j-1} \\ l_1 + \cdots + l_{k-j-1} = k-j+r}}  \frac{(k-j+r)!}{l_1! \cdots l_{k-j-1}!} \right) \\
&  + \sum_{\substack{u_1, \dots, u_j \geq 0 \\ u_1 + \cdots u_j = r+1 }} \pi_{u_1, \dots, u_j}  \\
&=   \frac{1}{(k-2)!} \left( \sum_{\substack{1 \leq l_1, \dots, l_{k-1} \\  l_1 + \cdots + l_{k-1} = k+r-1}} \frac{(k+r-1)!}{l_1! \cdots l_{k-1}!} + \sum_{\substack{1 \leq l_1, \dots, l_{k-2} \\  l_1 + \cdots + l_{k-2} = k+r-1}} \frac{(k+r-1)!}{l_1! \cdots l_{k-2}!}   \right) \\
& +  \sum_{l=k-1}^{k+r-1}  \binom{k+r}{l} \stirlingii{l}{k-1} \\
& + \frac{1}{(k-1)!} \sum_{i=1}^{r+1} (i+k-1)! \binom{n-k}{i} \sum_{l=i+k-1}^{k+r} \binom{k+r}{l} \stirlingii{l}{i+k-1} \\
& - \frac{1}{(k-j-1)!}  \sum_{i=1}^{r+1} (k-j-1+i)! \binom{n-k+j}{i} \sum_{l=i+k-2}^{k+r-1} \stirlingii{l-j+2}{k-j+i}\sum_{\substack{u_1, \dots, u_j \geq 0 \\ u_1 + \cdots u_j = k+r-l-1}} \pi_{u_1, \dots, u_j} \\
&-  \sum_{l=1}^r \stirlingii{k-j+l}{k-j} \sum_{\substack{u_1, \dots, u_j \geq 0 \\ u_1 + \cdots u_j = r-l+1}} \pi_{u_1,  \dots, u_j},
\end{align*}
which, after using relation \eqref{stirling_multinom} between multinomial coefficients and Stirling numbers becomes:
\begin{align*}
& \frac{1}{(k-j-1)!} \left( \stirlingii{k-j+r}{k-j}(k-j)! + \stirlingii{k-j+r}{k-j-1}(k-j-1)! \right) \\
&  + \sum_{\substack{u_1, \dots, u_j \geq 0 \\ u_1 + \cdots u_j = r+1 }} \pi_{u_1, \dots, u_j}  \\
&=    \frac{1}{(k-2)!} \left( \stirlingii{k+r-1}{k-1}(k-1)! + \stirlingii{k+r-1}{k-2}(k-2)!   \right) \\
& +  \sum_{l=k-1}^{k+r-1}  \binom{k+r}{l} \stirlingii{l}{k-1} \\
& + \frac{1}{(k-1)!} \sum_{i=1}^{r+1} (i+k-1)! \binom{n-k}{i} \sum_{l=i+k-1}^{k+r} \binom{k+r}{l} \stirlingii{l}{i+k-1} \\
& - \frac{1}{(k-j-1)!}  \sum_{i=1}^{r+1} (k-j-1+i)! \binom{n-k+j}{i} \sum_{l=i+k-2}^{k+r-1} \stirlingii{l-j+2}{k-j+i}\sum_{\substack{u_1, \dots, u_j \geq 0 \\ u_1 + \cdots u_j = k+r-l-1}} \pi_{u_1, \dots, u_j} \\
&-  \sum_{l=1}^r \stirlingii{k-j+l}{k-j} \sum_{\substack{u_1, \dots, u_j \geq 0 \\ u_1 + \cdots u_j = r-l+1}} \pi_{u_1,  \dots, u_j},
\end{align*}
and, after applying Eq. \eqref{stirling_1}:
\begin{align*}
&  \stirlingii{k-j+r+1}{k-j}   + \sum_{\substack{u_1, \dots, u_j \geq 0 \\ u_1 + \cdots u_j = r+1 }} \pi_{u_1, \dots, u_j}  \\
&=     \stirlingii{k+r}{k-1} +  \sum_{l=k-1}^{k+r-1}  \binom{k+r}{l} \stirlingii{l}{k-1} \\
& + \frac{1}{(k-1)!} \sum_{i=1}^{r+1} (i+k-1)! \binom{n-k}{i} \sum_{l=i+k-1}^{k+r} \binom{k+r}{l} \stirlingii{l}{i+k-1} \\
& - \frac{1}{(k-j-1)!}  \sum_{i=1}^{r+1} (k-j-1+i)! \binom{n-k+j}{i} \sum_{l=i+k-2}^{k+r-1} \stirlingii{l-j+2}{k-j+i}\sum_{\substack{u_1, \dots, u_j \geq 0 \\ u_1 + \cdots u_j = k+r-l-1}} \pi_{u_1, \dots, u_j} \\
&-  \sum_{l=1}^r \stirlingii{k-j+l}{k-j} \sum_{\substack{u_1, \dots, u_j \geq 0 \\ u_1 + \cdots u_j = r-l+1}} \pi_{u_1,  \dots, u_j}.
\end{align*}
After combining the first and the second line on the right-hand side, we have:
\begin{align*}
& \stirlingii{k-j+r+1}{k-j}   + \sum_{\substack{u_1, \dots, u_j \geq 0 \\ u_1 + \cdots u_j = r+1 }} \pi_{u_1, \dots, u_j}  \\
&=    \frac{1}{(k-1)!} \sum_{i=0}^{r+1} (i+k-1)! \binom{n-k}{i} \sum_{l=i+k-1}^{k+r} \binom{k+r}{l} \stirlingii{l}{i+k-1} \\
& - \frac{1}{(k-j-1)!}  \sum_{i=1}^{r+1} (k-j-1+i)! \binom{n-k+j}{i} \sum_{l=i+k-2}^{k+r-1} \stirlingii{l-j+2}{k-j+i}\sum_{\substack{u_1, \dots, u_j \geq 0 \\ u_1 + \cdots u_j = k+r-l-1}} \pi_{u_1, \dots, u_j} \\
&-  \sum_{l=1}^r \stirlingii{k-j+l}{k-j} \sum_{\substack{u_1, \dots, u_j \geq 0 \\ u_1 + \cdots u_j = r-l+1}} \pi_{u_1,  \dots, u_j},
\end{align*}
Now we combine the last term on the left-hand side with the last term on the right-hand side and obtain:
\begin{align*}
& \stirlingii{k-j+r+1}{k-j}     \\
&=    \frac{1}{(k-1)!} \sum_{i=0}^{r+1} (i+k-1)! \binom{n-k}{i} \sum_{l=i+k-1}^{k+r} \binom{k+r}{l} \stirlingii{l}{i+k-1} \\
& - \frac{1}{(k-j-1)!}  \sum_{i=1}^{r+1} (k-j-1+i)! \binom{n-k+j}{i} \sum_{l=i+k-2}^{k+r-1} \stirlingii{l-j+2}{k-j+i}\sum_{\substack{u_1, \dots, u_j \geq 0 \\ u_1 + \cdots u_j = k+r-l-1}} \pi_{u_1, \dots, u_j} \\
&-  \sum_{l=0}^r \stirlingii{k-j+l}{k-j} \sum_{\substack{u_1, \dots, u_j \geq 0 \\ u_1 + \cdots u_j = r-l+1}} \pi_{u_1,  \dots, u_j},
\end{align*}
After moving $l$ up for $k-2$ and moving $\stirlingii{k-j+r+1}{k-j}$ to the right-hand side, our equality becomes
\begin{align*}
0 &=  \frac{1}{(k-1)!} \sum_{i=0}^{r+1} (i+k-1)! \binom{n-k}{i} \sum_{l=i+k-1}^{k+r} \binom{k+r}{l} \stirlingii{l}{i+k-1} \\
& - \frac{1}{(k-j-1)!}  \sum_{i=1}^{r+1} (k-j-1+i)! \binom{n-k+j}{i} \sum_{l=i+k-2}^{k+r-1} \stirlingii{l-j+2}{k-j+i}\sum_{\substack{u_1, \dots, u_j \geq 0 \\ u_1 + \cdots u_j = k+r-l-1}} \pi_{u_1, \dots, u_j} \\
&- \stirlingii{k-j+r+1}{k-j} - \sum_{l=k-2}^{r+k-2} \stirlingii{l-j+2}{k-j} \sum_{\substack{u_1, \dots, u_j \geq 0 \\ u_1 + \cdots u_j = k+r-l-1}} \pi_{u_1,  \dots, u_j}.
\end{align*}
Last two lines on the right-hand side can be combined into one sum. After doing that and moving $l$ in the last row up by one to match the row above it, then moving all binomial coefficients from the right side to the left, we get the following equality:
\begin{align*}
&\frac{1}{(k-1)!} \sum_{i=0}^{r+1} (i+k-1)! \binom{n-k}{i} \sum_{l=i+k-1}^{k+r} \binom{k+r}{l} \stirlingii{l}{i+k-1} \\
&= \frac{1}{(k-j-1)!}  \sum_{i=0}^{r+1} (k-j-1+i)! \binom{n-k+j}{i} \sum_{l=i+k-2}^{k+r-1} \stirlingii{l-j+2}{k-j+i}\sum_{\substack{u_1, \dots, u_j \geq 0 \\ u_1 + \cdots u_j = k+r-l-1}} \pi_{u_1, \dots, u_j}.
\end{align*}
By the definition of Stirling numbers of the second kind, it holds that $\stirlingii{a}{b} = 0$ for $a < b$. Here, it means that 
\[  \sum_{l=i+k-1}^{k+r} \binom{k+r}{l} \stirlingii{l}{i+k-1} =  \sum_{l=0}^{k+r} \binom{k+r}{l} \stirlingii{l}{i+k-1} = \stirlingii{k+r+1}{i+k}, \]
where we used the Eq. \eqref{stirling_2}. 
Now, our equality becomes:
\begin{align*}
& \frac{1}{(k-1)!} \sum_{i=0}^{r+1} (i+k-1)! \binom{n-k}{i} \stirlingii{k+r+1}{i+k} \\
&= \frac{1}{(k-j-1)!}  \sum_{i=0}^{r+1} (k-j-1+i)! \binom{n-k+j}{i} \sum_{l=i+k-2}^{k+r-1} \stirlingii{l-j+2}{k-j+i}\sum_{\substack{u_1, \dots, u_j \geq 0 \\ u_1 + \cdots u_j = k+r-l-1}} \pi_{u_1, \dots, u_j}.
\end{align*}
Directly from Lemma \ref{lema_induktivna_suma} we can conclude that
\begin{align*}
    \sum_{\substack{u_1, \dots, u_j \geq 0 \\ u_1 + \cdots u_j = k+r-l-1}} \pi_{u_1, \dots, u_j} &= \sum_{\substack{u_1, \dots, u_j \geq 0 \\ u_1 + \cdots u_j = k+r-l-1}} (n-k+j)^{u_1} (n-k+j-1)^{u_2}\cdots (n-k+j - (j-1))^{u_j} \\
    &= \sum_{s=0}^{k+r-l-1} \binom{n-k}{s} \frac{(s+j-1)!}{(j-1)!} \stirlingii{j+k+r-l-1}{s+j},
\end{align*}
which means that what is left to be proven is:
\begin{align*}
        &\frac{1}{(k-1)!}\sum_{i=0}^{r+1} (i+k-1)! \binom{n-k}{i} \stirlingii{k+r+1}{i+k} \nonumber \\
        & \qquad = \frac{1}{(k-j-1)!} \sum_{i=0}^{r+1} (k-j-1+i)! \binom{n-k+j}{i} \sum_{l=i+k-2}^{k+r-1} \stirlingii{l-j+2}{k-j+i} \nonumber \\
        & \qquad \qquad \qquad \qquad \qquad \times \sum_{s=0}^{k+r-l-1} \binom{n-k}{s} \frac{(s+j-1)!}{(j-1)!} \stirlingii{j+k+r-l-1}{s+j}.
    \end{align*}
After noting that in the second sum on the right-hand side $l$ can start either from $k-2$ or $i+k-2$, this becomes exactly the statement of Lemma \ref{lema_orig_teska}, and hence the proof of Theorem \ref{main_theorem} is completed.
\end{proof}

{
\begin{remark}
We note that the condition $k < (n+j)/2$ is sufficient to guarantee $C_1 \neq 0$, where $C_1$ is the constant from the proof. The theorem still holds under the weaker condition that  $C_1 \neq 0$ for each $r=0,1,\ldots$
\end{remark}
}


\section{Conclusion and outlook}\label{sec:conclusion}

In this final section, we discuss a potential application to goodness-of-fit testing. In particular, this class of characterizations belongs to the group of so-called equidistribution-type characterizations which, due to the variety of ways of expressing equality in distribution of two random variables, has proven very fruitful for the construction of goodness-of-fit tests (see \cite{nikitin2017tests} for more details). For more recent applications, we refer to, e.g., \cite{cuparic2022new}. 

Characterizations stemming from \eqref{identitet} have already been exploited in this context. In particular, the U-empirical and V-empirical distribution function approach has been used to construct test statistics from the following distributional characterizations in the aforementioned references:

\begin{itemize}
    \item $X_{(2;2)}+\frac13X_{3}\overset{d}{=} X_{(3;3)}$ \cite{volkova2015goodness};
    \item $\sum_{i=1}^n \frac{1}{i}X_i\overset{d}{=} X_{(n;n)}$ \cite{jovanovic2015tests};
    \item $X_{(2;3)}+X_0\overset{d}{=} X_{(3;3)}$ \cite{milovsevic2016asymptotic};	 
    \item $X_{(1;2)}+X_0\overset{d}{=} X_{(2;2)}$ \cite{milovsevic2016some2};
    \item $X_{(n-1;n-1)}+\frac1nX_{n}\overset{d}{=} X_{(n;n)}$ and $X_{(n-1;n)}+X_{0}\overset{d}{=} X_{(n;n)}$ \cite{milovsevic2017some}.
\end{itemize}

The corresponding tests have been shown to be scale-free under the null hypothesis of exponentiality, which makes them suitable for testing even when the scale parameter is unknown. Moreover, they have rather competitive powers. Therefore, one potential application of the presented characterization lies in that direction, possibly by also exploring other integral transforms, such as U-empirical and V-empirical Laplace, Hankel, and characteristic function transforms, among others.


\bibliography{literatura}

\end{document}